\documentclass{amsart}
\usepackage{amsmath}
\usepackage{amsfonts}

\newtheorem{theorem}{Theorem}
\theoremstyle{plain}

\newtheorem{corollary}{Corollary}

\newtheorem{definition}{Definition}

\newtheorem{lemma}{Lemma}

\numberwithin{equation}{section}
\input{tcilatex}

\begin{document}
\title[On the Modified $q$-Genocchi numbers and polynomials with weight $%
\left( \alpha ,\beta \right) $]{ A note on the\ modified $q$-Genocchi
numbers and polynomials with weight $\left( \alpha ,\beta \right) $ and
their interpolation function at negative integers}
\author{Serkan Arac\i }
\address{University of Gaziantep, Faculty of Science and Arts, Department of
Mathematics, 27310 Gaziantep, TURKEY}
\email{mtsrkn@hotmail.com}
\author{Mehmet A\c{c}\i kg\"{o}z}
\address{University of Gaziantep, Faculty of Science and Arts, Department of
Mathematics, 27310 Gaziantep, TURKEY}
\email{acikgoz@gantep.edu.tr}
\author{Feng Qi}
\address{Department of Mathematics, College of Science, Tianjin Polytechnic
University, Tianjin 300160, China}
\email{qifeng618@gmail.com}
\author{Hassan Jolany}
\address{School of Mathematics, Statistics and Computer Science, University
of Tehran, Iran }
\email{hassan.jolany@khayam.ut.ac.ir}
\date{December 12, 2011}

\begin{abstract}
The purpose of this paper concerns to establish modified $q$-Genocchi
numbers and polynomials with weight ($\alpha $,$\beta $). In this paper we
investigate special generalized $q$-Genocchi polynomials and we apply the
method of generating function, which are exploited to derive further classes
of $q$-Genocchi polynomials and develop $q$-Genocchi numbers and
polynomials. By using the Laplace-Mellin transformation integral, we define $%
q$-Zeta function with weight ($\alpha $,$\beta $) and by presenting a link
between $q$-Zeta function with weight ($\alpha $,$\beta $) and $q$-Genocchi
numbers with weight ($\alpha $,$\beta $) we obtain an interpolation formula
for the $q$-Genocchi numbers and polynomials with weight ($\alpha $,$\beta $%
). Also we derive distribution formula (Multiplication Theorem) and Witt's
type formula for modified $q$-Genocchi numbers and polynomials with weight ($%
\alpha $,$\beta $) which yields a deeper insight into the effectiveness of
this type of generalizations for $q$-Genocchi numbers and polynomials. Our
new generating function possess a number of interesting properties which we
state in this paper.
\end{abstract}

\maketitle

\section{Introduction, Definitions and Notations}

\bigskip Recently, $q$-calculus has served as a bridge between mathematics
and physics. Therefore, there is a significant increase of activity in the
area of the $q$-calculus due to applications of the $q$-calculus in
mathematics, statistics and physics. The majority of scientists in the world
who use $q$-calculus today are physicists. $q$-Calculus is a generalization
of many subjects, like hypergeometric series, generating functions, complex
analysis, and particle physics. In short, $q$-calculus is quite a popular
subject today. One of Important Branch of $q$-calculus in number theory is $%
q $-type of special generating functions, for instance $q$-Bernoulli
numbers, $q$-Euler numbers, and $q$-Genocchi numbers, here we introduce a
new class of $q$-type generating function. We introduce $q$-Genocchi numbers
with weight $\left( \alpha ,\beta \right) $. When we define a new class of
generating functions like, $q$-Genocchi numbers with weight $\left( \alpha
,\beta \right) $, then we face to with this question that \textquotedblleft
can we define a new $q$-Zeta type function in related of this new class of $%
q $-type generating function?\textquotedblright . We give a positive answer
for our new class of numbers and polynomials. More precisely we show that
our $q$-type generating function is generalization of the Hurwitz Zeta
function. Historically many authors have tried to give $q$-analogues of the
Riemann Zeta function $\zeta \left( s\right) $, and its related functions.
By just following the method of Kaneko et al. [M. Kaneko, N. Kurokawa and M.
Wakayama, A variation of Euler's approach to the Riemann Zeta function,
Kyushu J. Math. 57 (2003), 175--192], who mainly used Euler-Maclaurin
summation formula to present and investigate a $q$-analogue of the Riemann
zeta function $\zeta \left( s\right) $, and gave a good and reasonable
explanation that their $q$-analogue may be a best choice. They also
commented that $q$-analogue of $\zeta \left( s\right) $ can be achieved by
modifying their method. Furthermore it is clear that $q$-Genocchi
polynomials of weight $\left( \alpha ,\beta \right) $ are in a class of
orthogonal polynomials and we know that most such special functions that are
orthogonal are satisfied in multiplication theorem, so in this present paper
we show this property is true for $q$-Genocchi polynomials of weight $\left(
\alpha ,\beta \right) $. In this introductory section, we present the
definitions and notations (and some of the Important properties and
characteristics) of the various special functions, polynomials and numbers,
which are potentially useful in the remainder of the paper.

Assume that $p$ be a fixed odd prime number. Throughout this paper we use
the following notations. By $%
\mathbb{Z}
_{p}$ we denote the ring of $p$-adic rational integers, $%
\mathbb{Q}
$ denotes the field of rational numbers, $%
\mathbb{Q}
_{p}$ denotes the field of $p$-adic rational numbers, and $%
\mathbb{C}
_{p}$ denotes the completion of algebraic closure of $%
\mathbb{Q}
_{p}$. Let $%
\mathbb{N}
$ be the set of natural numbers and $%
\mathbb{Z}
_{+}=%
\mathbb{N}
\cup \left\{ 0\right\} .$ Let $v_{p}$ be the normalized exponential
valuation of $%
\mathbb{C}
_{p}$ with $\left\vert p\right\vert _{p}=p^{-v_{p}\left( p\right) }=p^{-1}.$
When one speaks of $q$-extension, $q$ is considered in many ways such as an
indeterminate, a complex number $q\in 
\mathbb{C}
$ or $p$-adic number $q\in 
\mathbb{C}
_{p}.$ If $q\in 
\mathbb{C}
$ one normally assume that $\left\vert q\right\vert <1.$ If $q\in 
\mathbb{C}
_{p},$ we assume that $\left\vert 1-q\right\vert _{p}<p^{-\frac{1}{p-1}}$ so
that $q^{x}=\exp \left( x\log q\right) $ for $\left\vert x\right\vert
_{p}\leq 1.$ We use the following notation as follows:

\begin{equation*}
\left[ x\right] _{q}=\frac{1-q^{x}}{1-q}\text{, }\left[ x\right] _{-q}=\frac{%
1-\left( -q\right) ^{x}}{1+q}
\end{equation*}

Note that $\lim_{q\rightarrow 1}\left[ x\right] _{q}=x$; cf. [1-24].

For a fixed positive integer $d$ with $\left( d,f\right) =1,$ we set 
\begin{eqnarray*}
X &=&X_{d}=\lim_{\overleftarrow{N}}%
\mathbb{Z}
/dp^{N}%
\mathbb{Z}
, \\
X^{\ast } &=&\underset{\underset{\left( a,p\right) =1}{0<a<dp}}{\cup }a+dp%
\mathbb{Z}
_{p}
\end{eqnarray*}

and%
\begin{equation*}
a+dp^{N}%
\mathbb{Z}
_{p}=\left\{ x\in X\mid x\equiv a\left( \func{mod}dp^{N}\right) \right\} ,
\end{equation*}

where $a\in 
\mathbb{Z}
$ satisfies the condition $0\leq a<dp^{N}.$

By use Koblitz [N. Koblitz, $p$-adic Numbers $p$-adic Analysis and Zeta
Functions, Springer-Verlag, New York Inc, 1977] notations, A $p$-adic
distribution $\mu $ on $X$ is a $%
\mathbb{Q}
_{p}$-linear vector space homomorphism from the $%
\mathbb{Q}
_{p}$-vector space of locally constant functions on $X$ to $%
\mathbb{Q}
_{p}$. If \ $f:X\rightarrow 
\mathbb{Q}
_{p}$ is locally constant, instead of writing $\mu \left( f\right) $ for the
value of $\mu $ at $f$, we usually write $\int f\mu $. Also it is known that
we can write $\mu _{q}$ as follows: 
\begin{equation*}
\mu _{q}\left( x+p^{N}%
\mathbb{Z}
_{p}\right) =\frac{q^{x}}{\left[ p^{N}\right] _{q}}
\end{equation*}

is a distribution on $X$ for $q\in 
\mathbb{C}
_{p}$ with $\left\vert 1-q\right\vert _{p}\leq 1.$ For%
\begin{equation*}
f\in UD\left( 
\mathbb{Z}
_{p}\right) =\left\{ f\mid f:%
\mathbb{Z}
_{p}\rightarrow 
\mathbb{C}
_{p}\text{ is uniformly differentiable function}\right\} ,
\end{equation*}

the fermionic $p$-adic $q$-integral on $%
\mathbb{Z}
_{p}$ is defined by T. Kim as follows:%
\begin{eqnarray}
I_{-q}\left( f\right) &=&\int_{%
\mathbb{Z}
_{p}}f\left( x\right) d\mu _{-q}\left( x\right) =\lim_{N\rightarrow \infty
}\sum_{x=0}^{p^{N}-1}f\left( x\right) \mu _{-q}\left( x+p^{N}%
\mathbb{Z}
_{p}\right)  \label{equation 3} \\
&=&\lim_{N\rightarrow \infty }\frac{1}{\left[ p^{N}\right] _{-q}}%
\sum_{x=0}^{p^{N}-1}\left( -1\right) ^{x}f\left( x\right) q^{x}  \notag
\end{eqnarray}

Let $q\rightarrow 1,$ then we have fermionic integration on $%
\mathbb{Z}
_{p}$ as follows:%
\begin{equation*}
I_{-1}\left( f\right) =\int_{%
\mathbb{Z}
_{p}}f\left( x\right) d\mu _{-1}\left( x\right) =\lim_{N\rightarrow \infty
}\sum_{x=0}^{p^{N}-1}\left( -1\right) ^{x}f\left( x\right) ,
\end{equation*}

So by applying $f\left( x\right) =e^{xt},$ we get%
\begin{equation}
t\int_{%
\mathbb{Z}
_{p}}e^{tx}d\mu _{-1}\left( x\right) =\frac{2t}{e^{t}+1}=\sum_{n=0}^{\infty
}G_{n}\frac{t^{n}}{n!}  \label{equation 4}
\end{equation}

Where $G_{n}$ are Genocchi numbers. By using (\ref{equation 4}), we have%
\begin{equation*}
\int_{%
\mathbb{Z}
_{p}}e^{xt}d\mu _{-1}\left( x\right) =\sum_{n=0}^{\infty }\frac{G_{n+1}}{n+1}%
\frac{t^{n}}{n!}
\end{equation*}

so from above, we obtain 
\begin{equation*}
\sum_{n=0}^{\infty }\left( \int_{%
\mathbb{Z}
_{p}}x^{n}d\mu _{-1}\left( x\right) \right) \frac{t^{n}}{n!}%
=\sum_{n=0}^{\infty }\left( \frac{G_{n+1}}{n+1}\right) \frac{t^{n}}{n!}
\end{equation*}

By comparing coefficients of $\frac{t^{n}}{n!}$ on both sides of the above
equation it is fairly straightforward to deduce,%
\begin{equation*}
\frac{G_{n+1}}{n+1}=\int_{%
\mathbb{Z}
_{p}}x^{n}d\mu _{-1}\left( x\right) .
\end{equation*}

The definition of modified $q$-Euler numbers are given by 
\begin{equation}
\varepsilon _{0,q}=\frac{\left[ 2\right] _{q}}{2},\text{ }\left(
q\varepsilon +1\right) ^{k}-\varepsilon _{k,q}=\left\{ \QATOP{\left[ 2\right]
_{q},\text{ }k=0}{0,\text{ }k>0}\right.  \label{equation 6}
\end{equation}

with usual the convention about replacing $\varepsilon ^{k}$ by $\varepsilon
_{k,q}$ cf. \cite{kim 4},\cite{Ozden}. It was known that the modified $q$%
-euler numbers can be represented by $p$-adic $q$-integral on $%
\mathbb{Z}
_{p}$ as follows:%
\begin{equation*}
\varepsilon _{n,q}=\int_{%
\mathbb{Z}
_{p}}q^{-t}\left[ t\right] _{q}^{n}d\mu _{-q}\left( t\right) .
\end{equation*}

In [3,14,15,17], $q$-Genocchi numbers are defined as follows:

\begin{equation}
G_{0,q}=0,\text{ and }q\left( qG_{q}+1\right) ^{n}+G_{n,q}=\left\{ \QATOPD.
. {\left[ 2\right] _{q},n=1}{0,\text{ \ \ \ }n>1}\right.  \label{equation 2}
\end{equation}

with the usual convention of replacing $\left( G_{q}\right) ^{n}$ by $%
G_{n,q}.$

In \cite{araci 6}, $\left( h,q\right) $-Genocchi numbers are indicated as:%
\begin{equation*}
G_{0,q}^{\left( h\right) }=0,\text{ and }q^{h-2}\left( qG_{q}^{\left(
h\right) }+1\right) ^{n}+G_{n,q}^{\left( h\right) }=\left\{ \QATOP{\left[ 2%
\right] _{q},\text{ }n=1}{0,\text{ \ \ \ }n>1,}\right.
\end{equation*}

with the usual convention about replacing $\left( G_{q}^{\left( h\right)
}\right) ^{n}$ by $G_{n,q}^{\left( h\right) }.$

Recently, for $n\in 
\mathbb{Z}
_{+},$ Araci et al. are considered weighted $q$-Genocchi numbers by 
\begin{equation}
\widetilde{G}_{0,q}^{\left( \alpha \right) }=0,\text{ }q^{1-\alpha }\left( q%
\widetilde{G}_{q}^{\left( \alpha \right) }+1\right) ^{n}+\widetilde{G}%
_{n,q}^{\left( \alpha \right) }=\left\{ \QATOP{\left[ 2\right] _{q},\text{ }%
n=1}{0,\text{ \ \ \ \ \ }n\neq 1,}\right.  \label{equation 5}
\end{equation}

with the usual convention about replacing $\left( \widetilde{G}_{q}\right)
^{n}$ by $\widetilde{G}_{n,q}$ (for more information, see \cite{araci 1})

For $\alpha ,n\in 
\mathbb{Z}
_{+}$ and $h\in 
\mathbb{N}
,$ Araci et al. \cite{araci 2} defined weighted $\left( h,q\right) $%
-Genocchi numbers as follows:%
\begin{equation*}
\widetilde{G}_{n+1,q}^{\left( \alpha ,h\right) }=\int_{%
\mathbb{Z}
_{p}}q^{\left( h-1\right) x}\left[ x\right] _{q^{\alpha }}^{n}d\mu
_{-q}\left( x\right) .
\end{equation*}

\bigskip Taekyun Kim, by using $p$-adic $q$-integral on $%
\mathbb{Z}
_{p}$, introduced a new class of numbers and polynomials. He added a weight
on $q$-Bernoulli numbers and polynomials and defined $q$-Bernoulli numbers
with weight $\alpha $. He is given some interesting properties concerning $q$%
-Bernoulli numbers and polynomials with weight $\alpha $. After, by using $p$%
-adic $q$-integral on $%
\mathbb{Z}
_{p},$ several mathematicians started to study on this new branch of
generating function theory and extended most of the symmetric properties of $%
q$-Bernoulli numbers and polynomials to $q$-Bernoulli numbers and
polynomials with weight $\alpha $ (for more informations, see \cite{araci 1},%
\cite{araci 2},\cite{araci 4},\cite{Araci5},\cite{araci 6},\cite{kim 1},\cite%
{Kim 19},\cite{kim 20},\cite{kim 21},\cite{Kim 22},\cite{Kim 23}). With the
same motivation, we also introduce modified $q$-Genocchi numbers and
polynomials with weight $\left( \alpha ,\beta \right) .$ Also, we give some
interesting properties this type of polynomials. Furthermore, we derive the $%
q$-extensions of zeta type functions with weight $\left( \alpha ,\beta
\right) $ from the Mellin transformation to this generating function which
interpolates the $q$-Genocchi polynomials with weight $\left( \alpha ,\beta
\right) $ at negative integers.

\section{\qquad Modified $q$-Genocchi numbers and polynomials with weight $%
\left( \protect\alpha ,\protect\beta \right) $}

In this section, we derive some interesting properties Modified $q$-Genocchi
numbers and polynomials with weight $\left( \alpha ,\beta \right) $.

\begin{lemma}
For $n\in 
\mathbb{Z}
_{+},$we obtain%
\begin{equation}
I_{-q}^{\left( \beta \right) }\left( q^{-\beta x}f_{n}\right) +\left(
-1\right) ^{n-1}I_{-q}^{\left( \beta \right) }\left( q^{-\beta x}f\right) = 
\left[ 2\right] _{q^{\beta }}\sum_{l=0}^{n-1}\left( -1\right)
^{n-l-1}f\left( l\right) ,  \label{equation 100}
\end{equation}
\end{lemma}

\begin{proof}
Let be $f_{n}\left( x\right) =f\left( x+n\right) $ and $I_{-q}^{\left( \beta
\right) }\left( f\right) =\int_{%
\mathbb{Z}
_{p}}f\left( x\right) d\mu _{-q^{\beta }}\left( x\right) $ \ by the (\ref%
{equation 3}), we easily get 
\begin{eqnarray}
-I_{-q}^{\left( \beta \right) }\left( q^{-\beta x}f_{1}\right)
&=&\lim_{N\rightarrow \infty }\frac{1}{\left[ p^{N}\right] _{-q^{\beta }}}%
\sum_{x=0}^{p^{N}-1}f\left( x+1\right) \left( -1\right) ^{x}  \notag \\
&=&\lim_{N\rightarrow \infty }\frac{1}{\left[ p^{N}\right] _{-q^{\beta }}}%
\sum_{x=0}^{p^{N}-1}f\left( x\right) \left( -1\right) ^{x}-\left[ 2\right]
_{q^{\beta }}\lim_{N\rightarrow \infty }\frac{f\left( p^{N}\right) +f\left(
0\right) }{1+q^{\beta p^{N}}}  \notag \\
&=&I_{-q}^{\left( \beta \right) }\left( q^{-\beta x}f\right) -\left[ 2\right]
_{q^{\beta }}f\left( 0\right)  \label{equation 34}
\end{eqnarray}%
and%
\begin{eqnarray*}
I_{-q}^{\left( \beta \right) }\left( q^{-\beta x}f_{2}\right) &=&\int_{%
\mathbb{Z}
_{p}}q^{-\beta x}f\left( x+2\right) d\mu _{-q^{\beta }}\left( x\right)
=\lim_{N\rightarrow \infty }\frac{1}{\left[ p^{N}\right] _{-q^{\beta }}}%
\sum_{x=0}^{p^{N}-1}f\left( x+2\right) \left( -1\right) ^{x} \\
&=&I_{-q}^{\left( \beta \right) }\left( q^{-\beta x}f\right) +\left[ 2\right]
_{q^{\beta }}\lim_{N\rightarrow \infty }\frac{-f\left( 0\right) +f\left(
1\right) -f\left( p^{N}\right) +f\left( p^{N}+1\right) }{1+q^{\beta p^{N}}}
\\
&=&I_{-q}^{\left( \beta \right) }\left( q^{-\beta x}f\right) +\left[ 2\right]
_{q^{\beta }}\left( f\left( 1\right) -f\left( 0\right) \right)
\end{eqnarray*}%
Thus, we have%
\begin{equation*}
I_{-q}^{\left( \beta \right) }\left( q^{-\beta x}f_{2}\right)
-I_{-q}^{\left( \beta \right) }\left( q^{-\beta x}f\right) =\left[ 2\right]
_{q^{\beta }}\sum_{l=0}^{1}\left( -1\right) ^{1-l}f\left( l\right)
\end{equation*}

By continuing this process, we arrive at the desired result.
\end{proof}

\begin{definition}
Let $\alpha ,n,\beta \in 
\mathbb{Z}
_{+}.$ We define modified $q$-Genocchi numbers with weight $\left( \alpha
,\beta \right) $ as follows:%
\begin{equation}
\frac{g_{n+1,q}^{\left( \alpha ,\beta \right) }}{n+1}=\left[ 2\right]
_{q^{\beta }}\sum_{m=0}^{\infty }\left( -1\right) ^{m}\left[ m\right]
_{q^{\alpha }}^{n}  \label{equation 101}
\end{equation}
\end{definition}

\begin{theorem}
For $\alpha ,n,\beta \in 
\mathbb{Z}
_{+},$ we get%
\begin{equation}
\frac{g_{n+1,q}^{\left( \alpha ,\beta \right) }}{n+1}=\frac{\left[ 2\right]
_{q^{\beta }}}{\left( 1-q^{\alpha }\right) ^{n}}\sum_{l=0}^{n}\binom{n}{l}%
\left( -1\right) ^{l}\frac{1}{1+q^{\alpha l}}  \label{equation 102}
\end{equation}
\end{theorem}

\begin{proof}
By (\ref{equation 101}), we develop as follows: 
\begin{eqnarray*}
&&\frac{g_{n+1,q}^{\left( \alpha ,\beta \right) }}{n+1} \\
&=&\frac{\left[ 2\right] _{q^{\beta }}}{\left( 1-q^{\alpha }\right) ^{n}}%
\sum_{m=0}^{\infty }\left( -1\right) ^{m}\left( 1-q^{m\alpha }\right) ^{n} \\
&=&\frac{\left[ 2\right] _{q^{\beta }}}{\left( 1-q^{\alpha }\right) ^{n}}%
\sum_{m=0}^{\infty }\left( -1\right) ^{m}\sum_{l=0}^{n}\binom{n}{l}\left(
-1\right) ^{l}\left( q^{m\alpha }\right) ^{l} \\
&=&\frac{\left[ 2\right] _{q^{\beta }}}{\left( 1-q^{\alpha }\right) ^{n}}%
\sum_{l=0}^{n}\binom{n}{l}\left( -1\right) ^{l}\sum_{m=0}^{\infty }\left(
-1\right) ^{m}q^{m\alpha l} \\
&=&\frac{\left[ 2\right] _{q^{\beta }}}{\left( 1-q^{\alpha }\right) ^{n}}%
\sum_{l=0}^{n}\binom{n}{l}\left( -1\right) ^{l}\frac{1}{1+q^{\alpha l}}.
\end{eqnarray*}

Thus, we complete the proof of Theorem.
\end{proof}

By the following Theorem, we get Witt's type formula of this type
polynomials.

\begin{theorem}
For $\beta ,\alpha ,n\in 
\mathbb{Z}
_{+},$ we get 
\begin{equation}
\frac{g_{n+1,q}^{\left( \alpha ,\beta \right) }}{n+1}=\int_{%
\mathbb{Z}
_{p}}q^{-\beta x}\left[ x\right] _{q^{\alpha }}^{n}d\mu _{-q^{\beta }}\left(
x\right) .  \label{equation 104}
\end{equation}
\end{theorem}

\begin{proof}
By using $p$-adic $q$-integral on $%
\mathbb{Z}
_{p}$, namely, replace $f(x)$ by $q^{-\beta x}\left[ x\right] _{q^{\alpha
}}^{n}$ and $\mu _{-q}\left( x+p^{N}%
\mathbb{Z}
_{p}\right) $ by $\mu _{-q^{\beta }}\left( x+p^{N}%
\mathbb{Z}
_{p}\right) $ into (\ref{equation 3}), we get%
\begin{eqnarray}
\int_{%
\mathbb{Z}
_{p}}q^{-\beta x}\left[ x\right] _{q^{\alpha }}^{n}d\mu _{-q^{\beta }}\left(
x\right) &=&\frac{1}{\left( 1-q^{\alpha }\right) ^{n}}\sum_{l=0}^{n}\binom{n%
}{l}\left( -1\right) ^{l}\int_{%
\mathbb{Z}
_{p}}q^{\alpha lx-\beta x}d\mu _{-q^{\beta }}\left( x\right)  \notag \\
&=&\frac{1}{\left( 1-q^{\alpha }\right) ^{n}}\sum_{l=0}^{n}\binom{n}{l}%
\left( -1\right) ^{l}\lim_{N\rightarrow \infty }\frac{1}{\left[ p^{N}\right]
_{-q^{\beta }}}\sum_{x=0}^{p^{N}-1}\left( -q^{\alpha l}\right) ^{x}  \notag
\\
&=&\frac{1}{\left( 1-q^{\alpha }\right) ^{n}}\sum_{l=0}^{n}\binom{n}{l}%
\left( -1\right) ^{l}\frac{\left[ 2\right] _{q^{\beta }}}{1+q^{\alpha l}}%
\lim_{N\rightarrow \infty }\frac{1+\left( q^{\alpha l}\right) ^{p^{N}}}{%
1+q^{\beta p^{N}}}  \label{equation 103} \\
&=&\frac{\left[ 2\right] _{q^{\beta }}}{\left( 1-q^{\alpha }\right) ^{n}}%
\sum_{l=0}^{n}\binom{n}{l}\left( -1\right) ^{l}\frac{1}{1+q^{\alpha l}} 
\notag
\end{eqnarray}
Use of (\ref{equation 102}) and (\ref{equation 103}), we arrive at the
desired result.
\end{proof}

The Witt's type formula of modified $q$-Genocchi numbers with weight $\left(
\alpha ,\beta \right) $ asserted by Theorem 2, do aid in translating the
various properties and results involving $q$-Genocchi numbers with weight $%
\left( \alpha ,\beta \right) $ which we state some of them in this section.
We put $\alpha \rightarrow 1$ and $\beta \rightarrow 1$ into (\ref{equation
104}), we readily see $\frac{g_{n+1,q}^{\left( 1,1\right) }}{n+1}%
=\varepsilon _{n,q}.$

\begin{corollary}
Let $C_{q}^{\left( \alpha ,\beta \right) }\left( t\right)
=\sum_{n=0}^{\infty }$ $g_{n,q}^{\left( \alpha ,\beta \right) }\frac{t^{n}}{%
n!}.$ Then we have%
\begin{equation*}
C_{q}^{\left( \alpha ,\beta \right) }\left( t\right) =\left[ 2\right]
_{q^{\beta }}t\sum_{m=0}^{\infty }\left( -1\right) ^{m}e^{t\left[ m\right]
_{q^{\alpha }}}.
\end{equation*}
\end{corollary}

\begin{proof}
From (\ref{equation 101}) we easily get,%
\begin{equation}
\int_{%
\mathbb{Z}
_{p}}q^{-\beta x}e^{t\left[ x\right] _{q^{\alpha }}}d\mu _{-q^{\beta
}}\left( x\right) =\left[ 2\right] _{q^{\beta }}t\sum_{m=0}^{\infty }\left(
-1\right) ^{m}e^{t\left[ m\right] _{q^{\alpha }}}  \label{equation 105}
\end{equation}%
By expression (\ref{equation 105}), we have%
\begin{equation*}
\sum_{n=0}^{\infty }g_{n,q}^{\left( \alpha ,\beta \right) }\frac{t^{n}}{n!}=%
\left[ 2\right] _{q^{\beta }}t\sum_{m=0}^{\infty }\left( -1\right) ^{m}e^{t%
\left[ m\right] _{q^{\alpha }}}
\end{equation*}

Thus, we complete the proof of Theorem.
\end{proof}

Now, we consider the modified $q$-Genocchi polynomials polynomials with
weight $\alpha $\ as follows:%
\begin{equation}
\frac{g_{n+1,q}^{\left( \alpha ,\beta \right) }(x)}{n+1}=\int_{%
\mathbb{Z}
_{p}}q^{-\beta t}\left[ x+t\right] _{q^{\alpha }}^{n}d\mu _{-q^{\beta
}}\left( t\right) ,\text{ \ }n\in 
\mathbb{N}
\text{ and }\alpha \in 
\mathbb{Z}
_{+}  \label{equation 106}
\end{equation}

From expression (\ref{equation 106}), we see readily%
\begin{eqnarray}
\frac{g_{n+1,q}^{\left( \alpha ,\beta \right) }(x)}{n+1} &=&\frac{\left[ 2%
\right] _{q^{\beta }}}{\left( 1-q^{\alpha }\right) ^{n}}\sum_{l=0}^{n}\binom{%
n}{l}\left( -1\right) ^{l}q^{\alpha lx}\frac{1}{1+q^{\alpha l}}  \notag \\
&=&\left[ 2\right] _{q^{\beta }}\sum_{m=0}^{\infty }\left( -1\right) ^{m}%
\left[ m+x\right] _{q^{\alpha }}^{n}  \label{equation 107}
\end{eqnarray}%
Let $C_{q}^{\left( \alpha ,\beta \right) }\left( t,x\right)
=\sum_{n=0}^{\infty }g_{n,q}^{\left( \alpha ,\beta \right) }(x)\frac{t^{n}}{%
n!}.$ Then we have%
\begin{eqnarray}
C_{q}^{\left( \alpha ,\beta \right) }\left( t,x\right) &=&\left[ 2\right]
_{q^{\beta }}t\sum_{m=0}^{\infty }\left( -1\right) ^{m}e^{t\left[ m+x\right]
_{q^{\alpha }}}  \notag \\
&=&\sum_{n=0}^{\infty }g_{n,q}^{\left( \alpha ,\beta \right) }\left(
x\right) \frac{t^{n}}{n!}.  \label{equation 108}
\end{eqnarray}

By Lemma 1, we get the following Theorem:

\begin{theorem}
For $m\in 
\mathbb{N}
,$and $\alpha ,\beta ,n\in 
\mathbb{Z}
_{+},$ we get%
\begin{equation*}
\frac{g_{m+1,q}^{\left( \alpha ,\beta \right) }}{m+1}+\left( -1\right) ^{n-1}%
\frac{g_{m+1,q}^{\left( \alpha ,\beta \right) }\left( n\right) }{m+1}=\left[
2\right] _{q^{\beta }}\sum_{l=0}^{n-1}\left( -1\right) ^{n-l-1}\left[ l%
\right] _{q^{\alpha }}^{m}
\end{equation*}
\end{theorem}

\begin{proof}
By applying Lemma 1 the methodology and techniques used above in getting
some identities for the generating functions of the modified $q$-Genocchi
numbers and polynomials with weight $\left( \alpha ,\beta \right) ,$ we
arrive at the desired result.
\end{proof}

\begin{theorem}
The following identity holds:%
\begin{equation*}
g_{0,q}^{\left( \alpha ,\beta \right) }=0,\text{ and \ }g_{n,q}^{\left(
\alpha ,\beta \right) }\left( 1\right) +g_{n,q}^{\left( \alpha ,\beta
\right) }=\left\{ \QATOP{\left[ 2\right] _{q^{\beta }},\text{if }n=1,}{0,%
\text{ if }n>1.}\right.
\end{equation*}
\end{theorem}

\begin{proof}
In (\ref{equation 34}) it is known that%
\begin{equation*}
I_{-q}^{\left( \beta \right) }\left( q^{-\beta x}f_{1}\right)
+I_{-q}^{\left( \beta \right) }\left( q^{-\beta x}f\right) =\left[ 2\right]
_{q^{\beta }}f\left( 0\right)
\end{equation*}%
If we take $f(x)=e^{t\left[ x\right] _{q^{\alpha }}},$ then we have%
\begin{eqnarray}
\left[ 2\right] _{q^{\beta }} &=&\int_{%
\mathbb{Z}
_{p}}q^{-\beta x}e^{t\left[ x+1\right] _{q^{\alpha }}}d\mu _{-q^{\beta
}}\left( x\right) +\int_{%
\mathbb{Z}
_{p}}q^{-\beta x}e^{t\left[ x\right] _{q^{\alpha }}}d\mu _{-q^{-\beta
}}\left( x\right)  \notag \\
&=&\sum_{n=0}^{\infty }\left( g_{n,q}^{\left( \alpha ,\beta \right) }\left(
1\right) +g_{n,q}^{\left( \alpha ,\beta \right) }\right) \frac{t^{n-1}}{n!}
\label{equation 110}
\end{eqnarray}%
Therefore, we get the Proof of Theorem.
\end{proof}

\begin{theorem}
For $d\equiv 1\left( \func{mod}2\right) $, $\alpha ,\beta \in 
\mathbb{Z}
_{+}$ and $n\in 
\mathbb{N}
,$ we get, 
\begin{equation*}
g_{n,q}^{\left( \alpha ,\beta \right) }\left( dx\right) =\frac{\left[ d%
\right] _{q^{\alpha }}^{n-1}}{\left[ d\right] _{-q^{\beta }}}%
\sum_{a=0}^{d-1}\left( -1\right) ^{a}g_{n,q^{d}}^{\left( \alpha ,\beta
\right) }\left( x+\frac{a}{d}\right) .
\end{equation*}
\end{theorem}

\begin{proof}
From (\ref{equation 106}), we can easily derive the following (\ref{equation
111})%
\begin{eqnarray}
\int_{%
\mathbb{Z}
_{p}}q^{-\beta t}\left[ x+t\right] _{q^{\alpha }}^{n}d\mu _{-q^{\beta
}}\left( t\right)  &=&\frac{\left[ d\right] _{q^{\alpha }}^{n}}{\left[ d%
\right] _{-q^{\beta }}}\sum_{a=0}^{d-1}\left( -1\right) ^{a}\int_{%
\mathbb{Z}
_{p}}q^{-\beta t}\left[ \frac{x+a}{d}+t\right] _{q^{d\alpha }}^{n}d\mu
_{\left( -q^{d}\right) ^{\beta }}\left( t\right)   \notag \\
&=&\frac{\left[ d\right] _{q^{\alpha }}^{n}}{\left[ d\right] _{-q^{\beta }}}%
\sum_{a=0}^{d-1}\left( -1\right) ^{a}\frac{g_{n+1,q^{d}}^{\left( \alpha
,\beta \right) }\left( \frac{x+a}{d}\right) }{n+1}.  \label{equation 111}
\end{eqnarray}%
So, by applying expression (\ref{equation 111}), we get at the desired
result and proof is complete.
\end{proof}

\section{Interpolation function of the polynomials $g_{n,q}^{\left( \protect%
\alpha ,\protect\beta \right) }\left( x\right) $}

In this section, we derive the interpolation function of the generating
functions of modified $q$-Genocchi polynomials with weight $\alpha $ and we
give the value of $q$-extension zeta function with weight $\left( \alpha
,\beta \right) $ at negative integers explicitly. For $s\in 
\mathbb{C}
$, by applying the Mellin transformation to (\ref{equation 108}), we obtain%
\begin{eqnarray*}
\xi ^{\left( \alpha ,\beta \right) }\left( s,x\mid q\right) &=&\frac{1}{%
\Gamma \left( s\right) }\int_{0}^{\infty }t^{s-2}\left\{ -C_{q}^{\left(
\alpha ,\beta \right) }\left( -t,x\right) \right\} dt \\
&=&\left[ 2\right] _{q^{\beta }}\sum_{m=0}^{\infty }\left( -1\right) ^{m}%
\frac{1}{\Gamma \left( s\right) }\int_{0}^{\infty }t^{s-1}e^{-t\left[ m+x%
\right] _{q^{\alpha }}}dt
\end{eqnarray*}

where $\Gamma \left( s\right) $ is Euler gamma function. We have%
\begin{equation*}
\xi ^{\left( \alpha ,\beta \right) }\left( s,x\mid q\right) =\left[ 2\right]
_{q^{\beta }}\sum_{m=0}^{\infty }\frac{\left( -1\right) ^{m}}{\left[ m+x%
\right] _{q^{\alpha }}^{s}}
\end{equation*}

So, we define $q$-extension zeta function with weight $\left( \alpha ,\beta
\right) $ as follows:

\begin{definition}
For $s\in 
\mathbb{C}
$ \ and $\alpha ,\beta \in 
\mathbb{N}
,$ we have%
\begin{equation}
\xi ^{\left( \alpha ,\beta \right) }\left( s,x\mid q\right) =\left[ 2\right]
_{q^{\beta }}\sum_{m=0}^{\infty }\frac{\left( -1\right) ^{m}}{\left[ m+x%
\right] _{q^{\alpha }}^{s}}  \label{equation 112}
\end{equation}%
$\xi ^{\left( \alpha ,\beta \right) }\left( s,x\mid q\right) $ can be
continued analytically to an entire function.
\end{definition}

Observe that, if $q\rightarrow 1,$ then $\xi ^{\left( \alpha ,\beta \right)
}\left( s,x\mid 1\right) =\zeta \left( s,x\right) $ which is the Hurwitz-
Euler zeta functions. Relation between $\xi ^{\left( \alpha ,\beta \right)
}\left( s,x\mid q\right) $ \ and $g_{n,q}^{\left( \alpha ,\beta \right)
}\left( x\right) $ are given by the following theorem:

\begin{theorem}
For $\alpha ,\beta \in 
\mathbb{N}
$ and $n\in 
\mathbb{N}
,$ we get 
\begin{equation*}
\xi ^{\left( \alpha ,\beta \right) }\left( -n,x\mid q\right) =\frac{%
g_{n+1,q}^{\left( \alpha ,\beta \right) }\left( x\right) }{n+1}.
\end{equation*}
\end{theorem}

\begin{proof}
By substituting $s=-n$ into (\ref{equation 112}), we arrive at the desired
result.
\end{proof}

\end{document}